\documentclass{owrart}

\usepackage[colorlinks=true,urlcolor=blue,linkcolor=red,citecolor=magenta]{hyperref}

\usepackage{paralist,amsthm}
\theoremstyle{theorem}
\newtheorem{theorem}{Theorem}

\newtheorem{lemma}[theorem]{Lemma}
\newtheorem{corollary}[theorem]{Corollary}
\newcommand\R{\mathbb{R}}

\newcommand\Sym{\mathfrak{S}}
\newcommand\dist{\mathrm{dist}}

\begin{document}

\addtolength\textheight{1cm}

\begin{talk}{Florian Frick}
{Counterexamples to the topological Tverberg conjecture}
{Frick, Florian}

\smallskip

The ``topological Tverberg conjecture'' states that for given integers $r \ge 2$, $d \ge 1$, $N = (r-1)(d+1)$, and for any continuous map $f \colon \Delta_N \to \R^d$ from the $N$-simplex $\Delta_N$ into $\R^d$ there are $r$ pairwise disjoint faces $\sigma_1, \dots, \sigma_r$ of~$\Delta_N$ such that $f(\sigma_1) \cap \dots \cap f(\sigma_r) \neq \emptyset$.
This holds if $f$ is an affine map: this is a reformulation of Tverberg's original theorem \cite{tverberg:generalisation_radon}.
The conjecture for continuous $f$ was introduced, and proven for $r$ a prime, 
by B\'ar\'any, Shlosman and Sz\H{u}cs \cite{barany_schlosman_szucs:toplogical_tverberg}, and later extended to the case when $r$ is a prime power by \"Ozaydin \cite{oezaydin:equivariant}.
The conjecture is trivial for $d=1$. All other cases have remained open.
According to Matou\v{s}ek \cite[p.~154]{matousek:borsuk-ulam}, the validity of the conjecture for general $r$ is  one of the most challenging problems in topological combinatorics.

Here we prove the existence of counterexamples to the topological Tverberg conjecture for any $r$ that is not a power of a prime and dimensions $d \ge 3r+1$. Our construction builds on recent work of 
Mabillard and Wagner \cite{Mabillard:2014:ETP:2582112.2582134}, from which we first obtain counterexamples to $r$-fold versions of the van Kampen--Flores theorem. Counterexamples to the topological Tverberg conjecture are then obtained by an additional application of the constraint method of Blagojevi\'c, Ziegler and the author \cite{blagojevic2014tverberg}.

In the conference proceedings version \cite{Mabillard:2014:ETP:2582112.2582134} Mabillard and Wagner announced the generalized van Kampen theorem together with an extended sketch of its proof; a full version of the paper is forthcoming. To state the generalized van Kampen theorem, we first need to fix some notation. We refer to Matou\v{s}ek \cite{matousek:borsuk-ulam} for further explanations. For a simplicial complex $K$ denote by 
$$
K^{\times r}_{\Delta(2)} = \{(x_1,\dots,x_r) \in \sigma_1 \times \dots \times \sigma_r \: | \: \sigma_i \text{ face of } K, \sigma_i \cap \sigma_j = \emptyset \ \forall i \neq j\}
$$
the \emph{$2$-wise deleted product} of $K$ and by $K^{(d)}$ the \emph{$d$-skeleton} of $K$. The space 
$K^{\times r}_{\Delta(2)}$ is a polytopal cell complex in a natural way (its faces are products of simplices). Denote by $W_r$ the vector space $\{(x_1, \dots, x_r) \in \R^r \: | \: \sum x_i =0\}$ with the action by the symmetric group $\Sym_r$ that permutes coordinates.

\begin{theorem}[Mabillard \& Wagner {\cite[Theorem 3]{Mabillard:2014:ETP:2582112.2582134}}]\label{vankampen}
	Suppose that $r\ge 2$, $k \ge 3$, and let $K$ be a simplicial complex of dimension $(r-1)k$. Then the following statements are equivalent:
	\begin{compactenum}[\rm(i)]
		\item There exists an $\Sym_r$-equivariant map $K^{\times r}_{\Delta(2)} \to S(W_r^{\oplus rk})$.
		\item There exists a continuous map $f \colon K \to \R^{rk}$ such that for any $r$ pairwise disjoint faces $\sigma_1, \dots, \sigma_r$ of~$K$ we have $f(\sigma_1) \cap \dots \cap f(\sigma_r) = \emptyset$.
	\end{compactenum}
\end{theorem}

An important result on the existence of equivariant maps was shown by \"Ozaydin.

\begin{lemma}[\"Ozaydin {\cite[Lemma 4.1]{oezaydin:equivariant}}]\label{lem:equivariant}
	Let $d \ge 3$ and $G$ be a finite group. Let $X$ be a $d$-dimensional free $G$-CW complex and let $Y$ be a $(d-2)$-connected G-CW complex. There is a $G$-map $X \to Y$ if and only if there are $G_p$-maps $X \to Y$ for every Sylow $p$-subgroup $G_p$, $p$ prime.
\end{lemma}

\"Ozaydin uses this result to prove the existence of $\Sym_r$-equivariant maps $$(\Delta_{(r-1)(d+1)})^{\times r}_{\Delta(2)} \to S(W_r^{\oplus d})$$ for $r$ not a prime power. An initial motivation for Mabillard and Wagner was to use such a map to construct counterexamples to the topological Tverberg conjecture via $r$-fold versions of the Whitney trick. 
However, for their approach to work they need codimension $k \ge 3$. Here we first derive counterexamples to $r$-fold versions of the van Kampen--Flores theorem, which is a Tverberg-type statement with a bound on the dimension of faces, see Corollary \ref{cor}, from the result of Mabillard and Wagner and \"Ozaydin's work, and eventually obtain counterexamples to the topological Tverberg conjecture by a combinatorial reduction.

\begin{corollary}\label{cor}
	Let $r \ge 6$ be an integer that is not a prime power and $k \ge 3$ an integer. Then for any $N$ there exists a continuous map $f\colon \Delta_N \to \R^{rk}$ such that for any $r$ pairwise disjoint faces $\sigma_1, \dots, \sigma_r$ of~$\Delta_N$ with $\dim \sigma_i \le (r-1)k$ we have $f(\sigma_1) \cap \dots \cap f(\sigma_r) = \emptyset$.
\end{corollary}

\begin{proof}
	Let $K = \Delta^{((r-1)k)}_N$ denote the $((r-1)k)$-dimensional skeleton of the simplex $\Delta_N$ on $N+1$ vertices. We only need to construct $f$ on $K$ and extend continuously to $\Delta_N$ in an arbitrary way. By Theorem \ref{vankampen} we need to show that there exists an $\Sym_r$-equivariant map $K^{\times r}_{\Delta(2)} \to S(W_r^{\oplus rk})$. The reasoning is the same as in \cite[Proof of Theorem 4.2]{oezaydin:equivariant}: the free $\Sym_r$-space $K^{\times r}_{\Delta(2)}$ has dimension at most $d=r(r-1)k$, and $S(W_r^{\oplus rk}) \cong S^{(r-1)rk-1}$ is $(d-2)$-connected. By Lemma \ref{lem:equivariant} the existence of an $\Sym_r$-map $K^{\times r}_{\Delta(2)} \to S(W_r^{\oplus rk})$ reduces to the existence of equivariant maps for Sylow $p$-subgroups, but $p$-groups have fixed points in $S(W_r^{\oplus rk})$ for $r$ not a prime power by \cite[Lemma 2.1]{oezaydin:equivariant}, so a constant map will do.
\end{proof}

The existence of the $\Sym_r$-equivariant map $K^{\times r}_{\Delta(2)} \to S(W_r^{\oplus rk})$ also follows immediately from \cite[Theorem 4.2]{oezaydin:equivariant} by observing that an $n$-dimensional, finite, free $\Sym_r$-complex always admits an equivariant map into an $(n-1)$-connected $\Sym_r$-space, see for example Matou\v{s}ek \cite[Lemma 6.2.2]{matousek:borsuk-ulam}.

Any $r$ generic affine subspaces of dimension $(r-1)k$ in $\R^{rk}$ intersect in a point by codimension reasons. 
Here we see that a continuous map $\Delta^{((r-1)k)}_N \to \R^{rk}$ can avoid this intersection, and indeed a map without
any such intersection exists for any $N$, but only if $r$ is not a prime power.
Volovikov \cite{volovikov:van-kampen_flores} proved that a map as postulated by 
Corollary \ref{cor} does not exist if $r$ is a prime power and $N \ge (r-1)(d+2)$ --- the case $r$ prime was proved by Sarkaria \cite{sarkaria1991}; see \cite{blagojevic2014tverberg} for more general results with significantly simplified proofs. 

The map $f$ in Corollary \ref{cor} could not be constructed if the topological Tverberg conjecture were true, since the validity of the topological Tverberg conjecture would imply such an intersection result for faces of bounded dimension by the constraint method. For the sake of completeness we will present a construction that does not rely on \cite{blagojevic2014tverberg}.

\begin{theorem}[The topological Tverberg conjecture fails]
	Let $r \ge 6$ be an integer that is not a prime power, and let $k \ge 3$ be an integer. Let $N =(r-1)(rk+2)$. Then there exists a continuous map $F \colon \Delta_N \to \R^{rk+1}$ such that for any $r$ pairwise disjoint faces $\sigma_1, \dots, \sigma_r$ of~$\Delta_N$ we have $F(\sigma_1) \cap \dots \cap F(\sigma_r) = \emptyset$.
\end{theorem}

\begin{proof}
	Let $f\colon \Delta_N \to \R^{rk}$ be a continuous map as constructed in Corollary \ref{cor}, that is, such that for any $r$ pairwise disjoint faces $\sigma_1, \dots, \sigma_r$ of~$\Delta_N$ with $\dim \sigma_i \le (r-1)k$ we have $f(\sigma_1) \cap \dots \cap f(\sigma_r) = \emptyset$.  Define $F\colon \Delta_N \to \R^{rk+1}$, $x \mapsto (f(x),\dist(x,\Delta_N^{((r-1)k)}))$. Suppose there were $r$ pairwise disjoint faces $\sigma_1, \dots, \sigma_r$ of~$\Delta_N$ such that there are points $x_i \in \sigma_i$ with $F(x_1) = \dots = F(x_r)$. By restricting to subfaces if necessary we can assume that $x_i$ is in the relative interior of $\sigma_i$. Then all the $x_i$ have the same distance to the $(r-1)k$-skeleton of $\Delta_N$. 
	
	Suppose all $\sigma_i$ had dimension at least $(r-1)k+1$. Then these faces would involve at least $r((r-1)k+2) = (r-1)(rk+2)+2>N+1$ vertices. Thus, one face $\sigma_j$ has dimension at most $(r-1)k$ and $\dist(x_j,\Delta_N^{((r-1)k)})= 0$. But then we have $\dist(x_i,\Delta_N^{((r-1)k)})= 0$ for all $i$, so $x_i\in\Delta_N^{((r-1)k)}$ and thus $\sigma_i\subseteq\Delta_N^{((r-1)k)}$ for all~$i$. This contradicts our assumption on $f$.
\end{proof}

If the topological Tverberg conjecture holds for $r$ pairwise disjoint faces and dimension $d+1$, then it also holds for dimension $d$ and the same number of faces. Thus, we are only interested in low-dimensional counterexamples. If $r$ is not a prime power then the topological Tverberg conjecture fails for dimensions $3r+1$ and above. Hence, the smallest counterexample this construction yields is a continuous map $\Delta_{100} \to \R^{19}$ such that any six pairwise disjoint faces have images that do not intersect in a common point.

\noindent
\emph{Update November 2015.} This and further applications of these methods can now be found in~\cite{blagojevic2015tverberg}. 
The full version of Mabillard and Wagner's extended abstract~\cite{Mabillard:2014:ETP:2582112.2582134} along with further constructions and a lower-dimensional counterexample is now available~\cite{Mabillard2015}.

\noindent
\emph{Acknowledgements.} I am grateful to Pavle Blagojevi\'c and G\"unter M. Ziegler for many insightful discussions. 
I would like to thank John M. Sullivan and Uli Wagner for improving the exposition of this manuscript with several good comments and suggestions.
Research supported by DFG via the Berlin Mathematical School.

\end{talk}


\begin{thebibliography}{99}
\footnotesize
%
%
%

\bibitem{barany_schlosman_szucs:toplogical_tverberg}
I. B{\'a}r{\'a}ny, S.~B. Shlosman, and A. Sz{\H{u}}cs, \emph{On a
  topological generalization of a theorem of {T}verberg}, J. Lond. Math. Soc.
  \textbf{2} (1981), no.~1, 158--164.
  
\bibitem{blagojevic2015tverberg}
P. V.~M. Blagojevi{\'c}, F. Frick, and G.~M. Ziegler, \emph{{Barycenters of Polytope Skeleta and Counterexamples to the Topological Tverberg Conjecture, via Constraints}}, Preprint, 6 pages, October 2015, \href{http://arxiv.org/abs/1510.07984}{arXiv:1510.07984}.

\bibitem{blagojevic2014tverberg}
\bysame,
  \emph{Tverberg plus constraints}, Bull. Lond. Math. Soc. \textbf{46} (2014),
  no.~5, 953--967.
  
\bibitem{Mabillard2015}
I. Mabillard and U. Wagner, \emph{{Eliminating Higher-Multiplicity Intersections, I. A Whitney
  Trick for Tverberg-Type Problems}}, Preprint, 46~pages,
  \href{http://arxiv.org/abs/1508.02349}{arXiv:1508.02349}, August 2015.

\bibitem{Mabillard:2014:ETP:2582112.2582134}
\bysame, \emph{{Eliminating Tverberg Points, I. An
  Analogue of the Whitney Trick}}, Proceedings of the Thirtieth Annual
  Symposium on Computational Geometry (New York, NY, USA), SOCG'14, ACM, 2014,
  pp.~171--180.

\bibitem{matousek:borsuk-ulam}
J. Matou{\v{s}}ek, \emph{{Using the Borsuk--Ulam Theorem. Lectures
  on Topological Methods in Combinatorics and Geometry}}, Universitext,
  {S}pringer-{V}erlag, 2003.

\bibitem{oezaydin:equivariant}
M. {\"O}zaydin, \emph{Equivariant maps for the symmetric group},
  unpublished, 17 pages, available online at
  http://minds.wisconsin.edu/handle/1793/63829,
  1987.
  
\bibitem{sarkaria1991}
P.~Sarkaria,
  \emph{A generalized van {K}ampen--{F}lores theorem}, Proc. Amer. Math. Soc. \textbf{111} (1991),
  no.~2, 559--565.

\bibitem{tverberg:generalisation_radon}
H. Tverberg, \emph{A generalization of {R}adon's theorem}, J. Lond. Math.
  Soc. \textbf{41} (1966), 123--128.

\bibitem{volovikov:van-kampen_flores}
A.~Yu. Volovikov, \emph{On the van {K}ampen--{F}lores theorem}, Math.
  Notes \textbf{59} (1996), no.~5, 477--481.

\end{thebibliography}
\end{document}